\documentclass[12pt,reqno]{amsart}
\usepackage[english]{babel}
\usepackage{amsmath,amsfonts,amssymb,amsthm,amscd,latexsym}
\usepackage[T1]{fontenc}
\usepackage[utf8]{inputenc}

\usepackage{tikz}
\usetikzlibrary{arrows}
\usepackage{color}
\usepackage{cite}
\usepackage{multirow}
\usepackage{cite}

\usepackage{ifpdf}
\ifpdf
 \usepackage[colorlinks,final,backref=page,hyperindex]{hyperref}
\else
  \usepackage[colorlinks,final,backref=page,hyperindex,hypertex]{hyperref}
\fi

\newtheorem{theorem}{Theorem}[section]
\newtheorem{lemma}[theorem]{Lemma}

\newtheorem{corollary}[theorem]{Corollary}
\theoremstyle{definition}

\theoremstyle{remark}
\newtheorem{remark}[theorem]{Remark}
\numberwithin{equation}{section}

\begin{document}

\setcounter{page}{1}

\title[Local and $2$-local automorphisms of simple generalized Witt algebras]{Local and $2$-local automorphisms of simple generalized Witt algebras}

\author{Yang Chen}
\address{Mathematics Postdoctoral Research Center,
Hebei Normal University, Shijiazhuang 050016, Hebei, China}
\email{chenyang1729@hotmail.com}

\author{Kaiming Zhao}
\address{Department of Mathematics, Wilfrid
Laurier University, Waterloo, ON, Canada N2L 3C5,  and School of Mathematical Sciences, Hebei Normal University, Shijiazhuang 050016, Hebei, China}
\email{kzhao@wlu.ca}

\author{Yueqiang Zhao}
\address{School of Mathematical Sciences, Hebei Normal University, Shijiazhuang 050016, Hebei, China}
\email{yueqiangzhao@163.com}

\date{}
\maketitle

\begin{abstract} In this paper, we prove that every invertible $2$-local or local automorphism of a simple generalized Witt algebra over any field of characteristic $0$ is an automorphism. In particular,  every $2$-local or local automorphism of Witt algebras $W_n$ is an automorphism for all $n\in \mathbb{N}$. But some simple generalized Witt algebras indeed  have  $2$-local (and  local) automorphisms that are not automorphisms.

\


\end{abstract}

\section{Introduction}

Let $\mathfrak{A}$ be an associative algebra. A linear operator $\Phi$ on $\mathfrak{A}$  is called a \textit{local automorphism} if for every $x\in \mathfrak{A}$ there exists an automorphism $\theta_x$ of $\mathfrak{A}$ , depending on $x$, such that $\Phi(x)= \theta_x(x)$. The concept of local automorphism was introduced   by 
Larson and Sourour \cite{LarSou} in 1990.
In \cite{LarSou} the authors actually proved that, invertible local automorphisms of the algebra of all bounded linear operators on an infinite-dimensional  Banach space  $X$ are  automorphisms, and the automorphisms and anti-automorphisms of the associative algebra $M_n(\mathbb{C})$ of complex $n\times n$ matrices exhaust all its local automorphisms. On the other hand, it was proved in \cite{Cri} that a commutative subalgebra of $M_3(\mathbb{C})$ has a local automorphism which is not an automorphism.

In 1997, \v{S}emrl \cite{Sem} introduced the notion of $2$-local automorphisms of algebras. Namely, a map $\Phi : \mathfrak{A}\to \mathfrak{A}$
(not necessarily linear) is called a \textit{$2$-local
automorphism} if for every $x,y\in \mathfrak{A}$, there exists an automorphism $\theta_{x,y} : \mathfrak{A}\to \mathfrak{A}$ such that
$\Phi(x)= \theta_{x,y} (x)$ and $\Phi(y)= \theta_{x,y}(y)$. Similarly, we have the definitions of local and $2$-local derivations. These concepts are actually important and interesting properties for an algebra.

Recently, several papers have been devoted to similar notions and corresponding problems for Lie (super)algebras $L$. The main problem in this subject is to determine all local and $2$-local automorphisms (resp. local and $2$-local derivations), and to see  whether every local or $2$-local automorphism (resp. local or $2$-local derivation) automatically becomes an automorphism (resp. a derivation) of $L$,  that is, whether automorphisms  (resp.  derivations) of an   algebra can be  completely determined by their local actions. In \cite{CW}, Chen and Wang initiated study of $2$-local automorphisms of finite-dimensional Lie algebras. They proved that if $L$ is a simple Lie algebra of type $A_l\ (l\geq 1)$, $D_l\ (l\geq 4)$ or $E_k\ (k= 6,7,8)$ over an algebraically closed field of characteristic zero, then every $2$-local automorphism of $L$ is an automorphism. This result was extended to any finite dimensional semisimple Lie algebra in \cite{AyuKudRak}.

 For local automorphisms of Lie algebras it was only known that the automorphisms and the anti-automorphisms of finite dimensional simple Lie algebra exhaust all its local automorphisms in \cite{Cos}. For infinite dimensional Lie algebras, $2$-local derivations on Witt algebras were determined using different approaches in \cite{AKY, Zhao}. We determined all  local derivations on Witt algebras in \cite{CZZ}. The present paper proves that every invertible $2$-local or local automorphism of a simple generalized Witt algebra is an automorphism.

Witt algebras were one of the four classes of Cartan type Lie algebras originally introduced in 1909 by Cartan \cite{C} when he studied
infinite dimensional simple Lie algebras over complex numbers. Generalized Witt algebras were defined by Kaplansky \cite{Kap} in the context of the classification problem of simple finite dimensional Lie algebras over fields of prime characteristic. The definition of generalized Witt algebras over fields of characteristic 0 was given by N. Kawamoto \cite{Kaw}. Using different notations from Kawamoto's, Djokovic and Zhao \cite{DZ} gave an essentially equivalent definition of generalized Witt algebras. Over the last two decades, the  representation theory of generalized Witt algebras  over complex numbers was  extensively
studied by many mathematicians and physicists; see for example \cite{BF, BMZ, GLLZ}. Very recently,  Billig and Futorny
\cite{BF}   obtained the classification for all simple Harish-Chandra modules over Lie algebra $W_n$ of vector fields on an $n$-dimensional torus, a well-known generalized Witt algebra.

The paper is organized as follows. In Section 2 we recall some known results  and establish some related properties concerning generalized Witt algebras  over any field of characteristic $0$. In Section 3 we prove that every invertible $2$-local or local automorphism of simple generalized Witt algebra is an automorphism.  In particular,  every $2$-local or local automorphism of Witt algebras $W_n$ is an automorphism for all $n\in \mathbb{N}$. But some simple generalized Witt algebras indeed  have  $2$-local (and  local) automorphisms that are not automorphisms. 

Throughout this paper, we denote by $\mathbb{Z}$, $\mathbb{N}$ and $\mathbb{C}$ the sets of all integers, positive integers and complex numbers respectively.
\medskip

\section{The generalized Witt algebras}

\medskip

In this section we recall definitions, symbols and establish some auxiliary
results for later use in this paper.

Recall from \cite{DZ}. Let $A$ be an abelian group, $\mathbb{F}$ a field of characteristic $0$, and $T$ a vector space over $\mathbb{F}$. The group algebra $\mathbb{F}A$ of $A$ over $\mathbb{F}$ is spanned by basis elements $t^\alpha$, $\alpha\in A$, and the multiplication of $\mathbb{F}A$ is defined by $t^\alpha t^\beta= t^{\alpha+ \beta}$, $\alpha, \beta\in A$. We shall write $1$ instead of $t^0$. The tensor product $W= \mathbb{F}A\otimes_{\mathbb{F}}T$ is a free left $\mathbb{F}A$-module. For the sake of simplicity, we shall write $t^\alpha\partial$ instead of $t^\alpha\otimes\partial$, $\alpha\in A$, $\partial\in T$. If a given mapping $\varphi: T\times A\rightarrow \mathbb{F}$ is $\mathbb{F}$-linear in the first variable and additive in the second one, then the bracket
$$[t^\alpha\partial_\alpha,t^\beta\partial_\beta]:= t^{\alpha+\beta}(\varphi(\partial_\alpha, \beta)\partial_\beta- \varphi(\partial_\beta, \alpha)\partial_\alpha),\ \alpha,\beta\in A,\ \partial_\alpha,\partial_\beta\in T$$
defines an infinite dimensional Lie algebra on the tensor product $W$. We shall refer to $W= W(A,T,\varphi)$ as a \textit{generalized Witt algebra}. We now introduce an $A$-gradation of $W$ by setting $W_\alpha= t^\alpha T$ for $\alpha\in A$. This gradation is compatible with the Lie algebra structure, i.e., $[W_\alpha, W_\beta]\subseteq W_{\alpha+ \beta}$ for all $\alpha, \beta\in A$. Consequently $T= W_0$ is the Cartan subalgebra in $W$. We say that $\varphi$ is \textit{nondegenerate} if
$$\varphi(\partial, \alpha)= 0,\ \forall\partial\in T\Rightarrow \alpha= 0$$
and
$$\varphi(\partial, \alpha)= 0,\ \forall\alpha\in A\Rightarrow \partial= 0.$$

The following theorem is due to Kawamoto \cite{Kaw}.
\begin{theorem}\label{thm21} \textit{Suppose that characteristic of $\mathbb{F}$ is $0$. Then $W= W(A,T,\varphi)$ is a simple Lie algebra if and only if $A\neq 0$ and $\varphi$ is nondegenerate.}
\end{theorem}

From now on, we shall assume that the characteristic of $\mathbb{F}$ is $0$ and $W(A,T,\varphi)$ is simple. Then $A$ must be torsion-free. If $A$ is finitely generated, then $A$ is isomorphic to $\mathbb{Z}^n$, $n\in \mathbb{N}$ with standard basis $\{\epsilon_1,\ldots,\epsilon_n\}$.
When we write $t_i$ instead of $t^{\epsilon_i}$, the group algebra $\mathbb{F}A$ becomes identified with the Laurent polynomial algebra $\mathbb{F}[t_1^{\pm 1},\ldots,t_n^{\pm 1}]$ over $\mathbb{F}$. Let $T$ be a $m$-dimensional vector space $V_m$, $m\in \mathbb{N}$
with a basis $\{\partial_1,\ldots,\partial_m\}$. Note that $n\geq m$, otherwise $\varphi$ is degenerate.
The corresponding generalized Witt algebra is denoted by $W(\mathbb{Z}^n,V_m,\varphi)$.
All generalized Witt algebras $W(\mathbb{Z}^n,V_n,\varphi)$ are isomorphic,
so we may define the mapping $\varphi:\ T\times A\rightarrow \mathbb{F}$ by setting $\varphi(\partial_i, \epsilon_i)= \delta_{ij}$. If we interpret $\partial_i$ as the differential operator
 $t_i\frac{\partial}{\partial t_i}$, then $W(\mathbb{Z}^n,V_n,\varphi)$ can be identified with the Lie algebra
$W_n= \text{Der}(\mathbb{F}[t_1^{\pm1},\ldots,t_n^{\pm1}])$ of derivations of the Laurent polynomial algebra $\mathbb{F}[t_1^{\pm1},\ldots,t_n^{\pm1}]$ over $\mathbb{F}$.
Specially, if $\mathbb{F}= \mathbb{C}$, then $W(\mathbb{Z}^n,V_n,\varphi)$ is
 the Lie algebra of vector fields on an $n$-dimensional torus.
 We can obtain similar definition and formula for the generalized Witt algebra $W_\infty= \rm{Der}(\mathbb{F}[t_1^\pm,t_2^\pm,\cdots])$.

An automorphism of a Lie algebra $L$ is an invertible linear map
$\theta: L\rightarrow L$ which satisfies
$$
\theta([x,y])=[\theta(x), \theta(y)],\
\forall x,y\in L.$$ The set of all automorphisms of
$L$ is denoted by $\text{Aut}(L)$.

In \cite[Section 5]{DZ}, Djokovic and Zhao gave explicit formulas of the automorphisms of simple generalized Witt algebras.

\begin{theorem}\label{thm22} \textit{Suppose that $W= W(A,T,\varphi)$ is simple and $\theta\in \rm{Aut}(W)$. Then
$$\theta(t^\alpha \partial_\alpha):= \theta(\chi,\sigma,\tau)(t^\alpha \partial_\alpha)= \chi(\alpha)t^{\sigma(\alpha)}\tau(\partial_\alpha),\ \alpha\in A,\ \partial_\alpha\in T,$$
where $\chi$ is a character of $A$, and $\sigma\in \rm{Aut}(A)$, $\tau\in \rm{GL}(T)$ satisfy $\varphi(\tau(\partial), \sigma(\gamma))= \varphi(\partial, \gamma)$ for any $\partial\in T,\ \gamma\in A$.}
\end{theorem}

We need the following result several times later.

\begin{lemma}\label{lem23} \textit{
Let $n\in \mathbb{N}$, and let  $k_1,\ldots,k_n,k'_1,\ldots,k'_n\in \mathbb{N}+1$ be such that any two of them are relatively prime. Let $S\subset S'=\{k_1\epsilon_1, k'_1\epsilon_1, \ldots, k_n\epsilon_n, k'_n\epsilon_n\}$  such that $S$ involves all $\epsilon_i$. If  $\sigma\in \rm{Aut}(\mathbb{Z}^n)$ satisfying $\sigma(S)\subset S'$, then $\sigma=\rm{id}_{\mathbb{Z}^n}$, the identity mapping on ${\mathbb{Z}^n}$.}
\end{lemma}

\begin{proof} For each $i$ we know that $k_i\epsilon_i\in S$ or $k'_i\epsilon_i\in S$. We may assume that
$k_i\epsilon_i\in S$ for all $i$ after switching some $k_j$ with $k_j'$ if necessary. Then  $\sigma(k_i\epsilon_i)=k_j \epsilon_j$ or $\sigma(k_i\epsilon_i)=k'_j \epsilon_j$ for some $j$. If $k_i\sigma(\epsilon_i)=k'_j \epsilon_j$
We see that $$\sigma(\epsilon_i)=\frac{k'_j}{k_i} \epsilon_j\notin \mathbb{Z}^n$$ which is impossible. If $k_i\sigma(\epsilon_i)=k_j \epsilon_j$ with $i\ne j$ we see that $$\sigma(\epsilon_i)=\frac{k_j}{k_i} \epsilon_j\notin \mathbb{Z}^n$$ which is impossible neither.
Thus we deduce that $\sigma(k_i\epsilon_i)=k_i \epsilon_i$  for all $i$, i.e.,
 $\sigma(\epsilon_i)=\epsilon_i$ for all $i$. Therefore $\sigma= \rm{id}_{\mathbb{Z}^n}$.
\end{proof}

For $W(\mathbb{Z}^n,V_m,\varphi)$, let us fix an element
\begin{equation}\label{w_n}w_{n,m}= \sum_{i=1}^n (t^{k_i\epsilon_i}+ t^{k'_i\epsilon_i})\partial_i, \end{equation}
where $\partial_i= \partial_m$ for $ i\geq m,$ and $k_1,\ldots,k_n,k'_1,\ldots,k'_n\in \mathbb{N}+1$ such that any two of them are relatively prime. We prove the following property first.

\begin{lemma}\label{lem24} \textit{
If there is a subalgebra $W(\mathbb{Z}^n,V_m,\varphi)$ in $W(A,T,\varphi)$, and $\theta= \theta(\chi,\sigma,\tau)\in {\rm Aut}(W(A,T,\varphi))$ satisfying $\theta(w_{n,m})= w_{n,m}$, then $\theta\big|_{W(\mathbb{Z}^n,V_m,\varphi)}= \rm{id}_{W(\mathbb{Z}^n,V_m,\varphi)}$.
}
\end{lemma}

\begin{proof}
Since $\theta(w_{n,m})= w_{n,m}$, we see that $\sigma$ is a bijective mapping from the set $\{k_1\epsilon_1, k'_1\epsilon_1, \ldots, k_n\epsilon_n, k'_n\epsilon_n\}$ to itself. Using Lemma \ref{lem23} we see that $\sigma\big|_{\mathbb{Z}^n}= \rm{id}_{\mathbb{Z}^n}$.

By Theorem \ref{thm22}, we have $\tau(\partial_i)\in\mathbb{C}\partial_i$ and
$$\varphi(\tau(\partial_i), \alpha)= \varphi(\tau(\partial_i), \sigma(\alpha))= \varphi(\partial_i, \alpha),\ \forall \alpha\in \mathbb{Z}^n,$$
which implies that $\tau(\partial_i)= \partial_i$, $i=1,\ldots,m$, i.e. $\tau\big|_{V_m}= \rm{id}_{V_m}$. Therefore
$$\theta(w_{n,m})= \sum_{i=1}^n (\chi(\epsilon_i)^{k_i}t^{k_i\epsilon_i}+ \chi(\epsilon_i)^{k'_i}t^{k'_i\epsilon_i})\partial_i,$$
yielding that $\chi(\epsilon_i)^{k_i}= \chi(\epsilon_i)^{k'_i}= 1$. Since $k_i$ is relatively prime to $k'_i$, it follows that $\chi(\epsilon_i)= 1$, $i=1,\ldots,n$, i.e. $\chi(\mathbb{Z}^n)= 1$. So $\theta\big|_{W(\mathbb{Z}^n,V_m,\varphi)}= \rm{id}_{W(\mathbb{Z}^n,V_m,\varphi)}$.
\end{proof}

\section{Local and $2$-local automorphisms of simple $W(A,T,\varphi)$}

In this section we shall determine all $2$-local and local automorphisms of the simple generalized Witt algebras $W(A,T,\varphi)$ over a field $\mathbb{F}$ of characteristic $0$. For $W(\mathbb{Z}^n,V_m,\varphi)$, we take
\begin{equation}\label{w_n}w_{n,m}= \sum_{i=1}^n (t^{k_i\epsilon_i}+ t^{k'_i\epsilon_i})\partial_i,\ w'_{n,m}= \sum_{i=1}^n (t^{-k_i\epsilon_i}+ t^{-k'_i\epsilon_i})\partial_i, \end{equation}
where  $\partial_i= \partial_m$ for $ i\geq m,$ and $k_1,\ldots,k_n,k'_1,\ldots,k'_n\in \mathbb{N}+1$ such that any two of them are relatively prime.

Let $A'$ be a subgroup of $A$ and $T'$ be a subspace of $T$. We say that $(A', T')$ is a {\bf non-degenerate pair} if
$$\varphi(\partial, A')= 0 \text{ for }\partial\in T'\Rightarrow \partial= 0$$
and
$$\varphi(T', \alpha)= 0 \text{ for }\alpha\in A'\Rightarrow \alpha= 0.$$

The following linear algebra result looks trivial. But we do not have a reference in hand.

\begin{lemma}\label{lem31'} \textit{Let $\bar{A}$ be a finitely generated subgroup of $A$ and $\bar{T}$ be a finite-dimensional subspace of $T$. Then there exist a finitely generated subgroup $A'$ of $A$ and  a finite-dimensional subspace $T'$ of $T$ such that $(A', T')$ is a non-degenerate pair with $\bar{A}\subset A'$ and
$\bar{T}\subset T'$.
}
\end{lemma}

\begin{proof} Let $ \bar T_{0}=\{\partial\in \bar{T}:\varphi(\partial, \bar{A})= 0 \}$ and
$\bar{A}_0=\{\alpha \in \bar{A}:\varphi(\bar{T}, \alpha)= 0 \}$. We shall prove the statement by induction on $r(\bar{A}, \bar{T})=\text{rank}\bar{A}_0+\dim \bar{T}_{0}$.

This is clear for $r(\bar{A}, \bar{T})=0$. Suppose that it is true for $r(\bar{A}, \bar{T})\le k\in \mathbb{N}$ and now suppose that $r(S, B)=k+1$. We have two cases to consider.

If  $\text{dim}\bar{T}_{0}\ge 1$, there is $\alpha\in A$ such that $\varphi(\bar{T}_{0}, \alpha)\ne0$. Take $\bar{A}'=\bar{A}+\alpha\mathbb{Z}$ and $\bar{T}'=\bar{T}$ in this case.

If  $\text{rank}\bar{A}_0\ge 1$, there is $\partial \in T$ such that $\varphi(\partial, \bar{A}_0)\ne0$. Take $\bar{T}'=\bar{T}+\mathbb{F}\partial$ and $\bar{A}'=\bar{A}$ in this case.

We see that in both cases $r(\bar{A}', \bar{T}')\le k$. By induction hypothesis the lemma follows.
\end{proof}

\begin{lemma}\label{lem31} \textit{
If $x,y$ are elements of the simple generalized Witt algebra $W(A,T,\varphi)$, then there exits $W(\mathbb{Z}^n,V_m,\varphi)$ such that
$$x,y\in W(\mathbb{Z}^n,V_m,\varphi)< W(A,T,\varphi),$$
where $<$ means to be a subalgebra.
}
\end{lemma}

\begin{proof}
Suppose that $x= \sum_{\alpha\in S}t^\alpha \partial_\alpha$ and $y= \sum_{\alpha\in S'}t^\alpha \partial'_\alpha$, where $S,S'$ are finite subsets of $A$. Let
$\bar{A}$ be the subgroup generated by $S$ and $S'$, $$\bar{T}=\text{span}\{ \partial_\alpha,  \partial'_{\alpha'}:
 \alpha\in S,  \alpha'\in S'\}.$$
By Lemma \ref{lem31'},   there exist    a   rank $n$ subgroup $A'$ of $A$ and  an $m$-dimensional subspace $T'$ of $T$ such that $(A', T')$ is a non-degenerate pair with $\bar{A}\subset A'$ and
$\bar{T}\subset T'$. We have the simple Lie algebra $ W(A',T',\varphi)\cong W(\mathbb{Z}^n,V_m,\varphi)$ and
$$x,y\in W(A',T',\varphi)\cong W(\mathbb{Z}^n,V_m,\varphi)< W(A,T,\varphi),\ m,n\in\mathbb{N}.$$
\end{proof}

\begin{theorem}\label{thm31} \textit{Every invertible $2$-local automorphism
of the simple generalized Witt algebra $W(A,T,\varphi)$ is an automorphism.}
\end{theorem}

\begin{proof}
Suppose that $\Phi$ is an invertible $2$-local automorphism of simple $W(A,T,\varphi)$. From Lemma \ref{lem31}, for any $x,y\in W(A,T,\varphi)$ there is  a simple subalgebra $$W(\mathbb{Z}^n,V_m,\varphi)<W(A,T,\varphi),\,\,m,n\in \mathbb{N}$$ such that $x,y\in W(\mathbb{Z}^n,V_m,\varphi)$.
For $w_{n,m}$ defined in (\ref{w_n}), there exists an automorphism $\theta_{w_{n,m},w_{n,m}}$ on $W(A,T,\varphi)$ such that $\Phi(w_{n,m})= \theta_{w_{n,m},w_{n,m}}(w_{n,m})$. Let $\Phi'= \theta_{w_{n,m},w_{n,m}}^{-1}\circ\Phi$,
then $\Phi'$ is a $2$-local automorphism of $W(A,T,\varphi)$ such that $\Phi'(w_{n,m})= w_{n,m}$. For any $z\in W(\mathbb{Z}^n,V_m,\varphi)$, there exists an automorphism $\theta_{w_{n,m},z}$ on $W(A,T,\varphi)$ such that $$w_{n,m}= \Phi'(w_{n,m})= \theta_{w_{n,m},z}(w_{n,m}),\text{ and } \Phi'(z)= \theta_{w_{n,m},z}(z).$$
It follows that $\theta_{w_{n,m},z}\big|_{W(\mathbb{Z}^n,V_m,\varphi)}= \rm{id}_{W(\mathbb{Z}^n,V_m,\varphi)}$ by Lemma \ref{lem24}, and thus $\Phi'(z)= z$, i.e. $\Phi'\big|_{W(\mathbb{Z}^n,V_m,\varphi)}=  \rm{id}_{W(\mathbb{Z}^n,V_m,\varphi)}$. Hence $$\Phi\big|_{W(\mathbb{Z}^n,V_m,\varphi)}= \theta_{w_{n,m},w_{n,m}}\big|_{W(\mathbb{Z}^n,V_m,\varphi)}.$$ It implies that $\Phi(cx+y)= c\Phi(x)+ \Phi(y)$, $c\in \mathbb{C}$ and $\Phi([x,y])= [\Phi(x), \Phi(y)]$. Therefore $\Phi$ is an automorphism of $W(A,T,\varphi)$.
\end{proof}

\begin{theorem}\label{thm32} \textit{Every invertible local automorphism
of the simple generalized Witt algebra $W(A,T,\varphi)$  over any field of characteristic $0$ is an automorphism.}
\end{theorem}

\begin{proof}
Suppose that $\Phi$ is a local automorphism of $W(A,T,\varphi)$.  From Lemma \ref{lem31}, for any $x,y\in W(A,T,\varphi)$ there is  a simple subalgebra $$W(\mathbb{Z}^n,V_m,\varphi)<W(A,T,\varphi),\,\,m,n\in \mathbb{N}$$ such that $x,y\in W(\mathbb{Z}^n,V_m,\varphi)$.
For $w_{n,m}$ defined in (\ref{w_n}),   there exists an automorphism $\theta_{w_{n,m}}$ on $W(A,T,\varphi)$ such that $\Phi(w_{n,m})= \theta_{w_{n,m}}(w_{n,m})$. Let $\Phi'= \theta_{w_{n,m}}^{-1}\circ\Phi$, then $\Phi'$ is a local automorphism of $W(A,T,\varphi)$ such that $\Phi'(w_{n,m})= w_{n,m}$.

For any $t^\alpha \partial_\alpha\in W(\mathbb{Z}^n,V_m,\varphi)$ with $\alpha\in \mathbb{Z}^n\setminus\cup_{i=1}^n\{k_i\epsilon_i, k'_i\epsilon_i\}$ and $\partial_\alpha\in V_m$, there exist automorphisms of $W(A,T,\varphi)$:
$$\theta_{t^\alpha \partial_\alpha}= \theta_{t^\alpha \partial_\alpha}(\chi_{t^\alpha \partial_\alpha},\sigma_{t^\alpha \partial_\alpha},\tau_{t^\alpha \partial_\alpha}),$$
$$\theta_{t^\alpha \partial_\alpha+ w_{n,m}}= \theta_{t^\alpha \partial_\alpha+ w_{n,m}}(\chi_{t^\alpha \partial_\alpha+ w_{n,m}},\sigma_{t^\alpha \partial_\alpha+ w_{n,m}},\tau_{t^\alpha \partial_\alpha+ w_{n,m}})$$
such that
$$\Phi'(t^\alpha \partial_\alpha)= \theta_{t^\alpha \partial_\alpha}(t^\alpha \partial_\alpha),\ \Phi'(t^\alpha \partial_\alpha+ w_{n,m})= \theta_{t^\alpha \partial_\alpha+ w_{n,m}}(t^\alpha \partial_\alpha+w_{n,m}).$$
Then we have
\begin{equation}\label{w_n''}\aligned\theta_{t^\alpha \partial_\alpha}(t^\alpha \partial_\alpha)+ w_{n,m}&= \Phi'(t^\alpha \partial_\alpha)+ w_{n,m}= \Phi'(t^\alpha \partial_\alpha)+ \Phi'(w_{n,m})\\
&= \Phi'(t^\alpha \partial_\alpha+ w_{n,m})= \theta_{t^\alpha \partial_\alpha+ w_{n,m}}(t^\alpha \partial_\alpha+w_{n,m})\\
&= \theta_{t^\alpha \partial_\alpha+ w_{n,m}}(t^\alpha \partial_\alpha)+\theta_{t^\alpha \partial_\alpha+ w_{n,m}}(w_{n,m}).\endaligned\end{equation}
It is clear that $\sigma_{t^\alpha \partial_\alpha+ w_{n,m}}$ is a bijective mapping from $\{\alpha, k_1\epsilon_1, k'_1\epsilon_1, \ldots, k_n\epsilon_n, k'_n\epsilon_n\}$ to $\{\sigma_{t^\alpha \partial_\alpha}(\alpha), k_1\epsilon_1, k'_1\epsilon_1, \ldots, k_n\epsilon_n, k'_n\epsilon_n\}$. It implies that $\sigma_{t^\alpha \partial_\alpha+ w_{n,m}}$ maps
$$\{\alpha, k_1\epsilon_1, k'_1\epsilon_1, \ldots, k_n\epsilon_n, k'_n\epsilon_n\}\setminus\{\alpha, \sigma_{t^\alpha \partial_\alpha+ w_{n,m}}^{-1}(\sigma_{t^\alpha \partial_\alpha}(\alpha))\}$$
to
$$\{ k_1\epsilon_1, k'_1\epsilon_1, \ldots, k_n\epsilon_n, k'_n\epsilon_n\}.$$
Using Lemma \ref{lem23} we see that  $\sigma_{t^\alpha \partial_\alpha+ w_{n,m}}\big|_{\mathbb{Z}^n}= \rm{id}_{\mathbb{Z}^n}$. Since  $\alpha\in \mathbb{Z}^n\setminus\cup_{i=1}^n\{k_i\epsilon_i, k'_i\epsilon_i\}$, from (\ref{w_n''})  we see  that $\theta_{t^\alpha \partial_\alpha+ w_{n,m}}(t^\alpha \partial_\alpha)= \theta_{t^\alpha \partial_\alpha}(t^\alpha \partial_\alpha)$. From (\ref{w_n''}) we deduce that $\theta_{t^\alpha \partial_\alpha+ w_{n,m}}(w_{n,m})= w_{n,m}$.
Therefore $\theta_{t^\alpha \partial_\alpha+ w_{n,m}}\big|_{W(\mathbb{Z}^n,V_m,\varphi)}=  \rm{id}_{W(\mathbb{Z}^n,V_m,\varphi)}$ by Lemma \ref{lem24}, and thus $$\Phi'(t^\alpha \partial_\alpha)= t^\alpha \partial_\alpha,\forall \partial_\alpha\in V_m,
\alpha\in \mathbb{Z}^n\setminus\cup_{i=1}^n\{k_i\epsilon_i, k'_i\epsilon_i\}.$$

Now  $\Phi'(w'_{n,m})= w'_{n,m}$.
In the arguments of the last paragraph if we replace $w_{n,m}$ with $w'_{n,m}$ and make corresponding modifications, we can prove that $$\Phi'(t^\alpha \partial_\alpha)= t^\alpha \partial_\alpha,\forall \partial_\alpha\in V_m, \alpha\in \mathbb{Z}^n\setminus\cup_{i=1}^n\{-k_i\epsilon_i, -k'_i\epsilon_i\}.$$
So $\Phi'\big|_{W(\mathbb{Z}^n,V_m,\varphi)}=  \rm{id}_{W(\mathbb{Z}^n,V_m,\varphi)}$, and thus $\Phi\big|_{W(\mathbb{Z}^n,V_m,\varphi)}= \theta_{w_{n,m}}\big|_{W(\mathbb{Z}^n,V_m,\varphi)}$. It implies that $\Phi([x,y])= [\Phi(x), \Phi(y)]$. Therefore $\Phi$ is an automorphism.
\end{proof}

\begin{remark}\label{rem43} \textit{Local and 2-local automorphisms must be injective by   definition but may not be surjective. For example, define the following linear map
$$\aligned \Phi: &W_\infty\rightarrow W_\infty,\\
&t^{\alpha}\frac d{dt_i}  \mapsto t^{(0,\alpha)}\frac d{dt_{i+1}},\ i\in \mathbb{N},\alpha\in\mathbb Z^\infty.\endaligned $$
Then $\Phi$ is not an automorphism since it is not surjective. However, $\Phi$ is both a local and 2-local automorphism. Therefore, it is necessary that we demand that the maps are invertible in Theorems \ref{thm31} and \ref{thm32}. But these conditions can be removed in some cases.}
\end{remark}

\begin{corollary}\label{col32} \textit{Suppose that  the group $A$ is finitely generated. Then
\begin{itemize}
\item[(a).] Every $2$-local automorphism
of the simple generalized Witt algebra $W(A,T,\varphi)$ is an automorphism;
\item[(b).]
Every   local automorphism
of the simple generalized Witt algebra $W(A,T,\varphi)$ is an automorphism.\end{itemize}}
\end{corollary}

\begin{proof} We may assume that $A\simeq \mathbb{Z}^n$ for some $n\in \mathbb{N}$. Then $\dim T=m\le n$. We see that $W(A,T,\varphi)= W(\mathbb{Z}^n,V_m,\varphi)$ for some $m$-dimensional vector space $V_m$. We may take $W(A,T,\varphi)$ directly as the subalgebra in the proofs of Theorems \ref{thm31} and \ref{thm32}, and thus the condition ``invertible'' is not needed in this case. The results follow.
\end{proof}

A special case of the above corollary is the following

\begin{corollary}\label{col32} \textit{ For any $n\in \mathbb{N}$, every $2$-local automorphism (or local automorphism)
of the Witt algebra  $W_n$    over a field of characteristic $0$    is an automorphism.}
\end{corollary}

\vspace{2mm}
\noindent
{\bf Acknowledgments. } This research is partially supported by NSFC (11871190) and NSERC (311907-2015).

\end{document}